\documentclass[a4paper,12pt]{article}

\usepackage{cmap}					% поиск в PDF
\usepackage{mathtext} 				% русские буквы в формулах
\usepackage[T2A]{fontenc}			% кодировка
\usepackage[utf8]{inputenc}			% кодировка исходного текста
\usepackage[english,russian]{babel}	% локализация и переносы
\usepackage{hyperref}
\usepackage{indentfirst}            % красная строка в первом абзаце
\usepackage{comment}
\usepackage{svg}
\usepackage{ifthen}
\usepackage{tikz}
\usepackage{float}

\usepackage{geometry}

\usetikzlibrary{math,calc,intersections}

\usepackage{amsmath,amsfonts,amssymb,amsthm,mathtools} % пакеты AMS

\renewcommand{\leq}{\ensuremath{\leqslant}}
\renewcommand{\geq}{\ensuremath{\geqslant}}

\newcommand{\eps}{\varepsilon}
\newcommand{\such}{\ \big|\ }

\newcommand{\N}{\mathbb{N}}

\newcommand{\R}{\mathbb{R}}
\renewcommand{\C}{\mathcal{C}}

\newcommand{\set}[1]{
	\left\{#1\right\}
}

\newcommand{\ps}[1]{
	\left(#1\right)
}

\newcommand{\sbr}[1]{
	\left[#1\right]
}

\newcommand{\lsi}[1]{
	\left[#1\right)
}

\newcommand{\rsi}[1]{
	\left(#1\right]
}

\newcommand{\md}[1]{
	\left|#1\right|
}

\DeclareMathOperator{\sign}{sign}

\theoremstyle{plain}
\newtheorem{theorem}{Теорема}
\newtheorem{lemma}{Лемма}
\newtheorem{corollary}{Следствие}
\newtheorem{proposition}{Утверждение}

\title{Нигде не монотонная интегрируемая по Риману производная, не имеющая локальных экстремумов}
\author{Н.А. Гусев, А.Е. Комагоров}

\begin{document}
	
	\maketitle
	
\begin{abstract}
	Построена нигде не монотонная интегрируемая по Риману производная, не имеющая локальных экстремумов.
	Также построена дифференцируемая функция $G$ с интегрируемой по Риману производной, имеющей изолированный локальный максимум в точке, не являющейся точкой перегиба функции $G$.
	
	We construct a nowhere monotone Riemann integrable derivative which has no local extrema. 
	We also construct a differentiable function $G$ such that $G'$ is Riemann integrable and has an isolated maximum which is not an inflection point of $G$.
\end{abstract}

В англоязычной Википедии \cite{Wiki} сформулировано следующее утверждение:
\begin{proposition}\label{fake-prop}
	Точка $x \in \R$ является точкой перегиба дифференцируемой функции~$f$ тогда и только тогда, когда $x$ является изолированным локальным экстремумом её производной.
%	The graph of the differentiable function has an inflection point at $(x, f(x))$ if and only if its first derivative $f'$ has an isolated extremum at $x$.
\end{proposition}
Нетрудно проверить, что данное утверждение справедливо для любой непрерывно дифференцируемой функции на интервале. 
Однако не всякая дифференцируемая функция является непрерывно дифференцируемой, что приводит к вопросу о справедливости утверждения~\ref{fake-prop} в общем случае.
Необходимость указанного условия для дифференцируемых функций сохраняется, так как строгая выпуклость (вогнутость) такой функции $f$ в односторонних окрестностях точки перегиба равносильна строгому возрастанию (убыванию) её производной, что влечёт отсутствие экстремумов у $f'$ в соответствующих окрестностях.
В~настоящей работе мы строим пример нигде не монотонной\footnotemark{} производной, имеющей изолированный экстремум. 
Данный пример показывает, что достаточность указанного в утверждении \ref{fake-prop} условия может нарушаться для дифференцируемых функций, не являющихся непрерывно дифференцируемыми.

\footnotetext{функция называется \emph{нигде не монотонной}, если она не является монотонной ни на одном интервале, лежащем в её области определения.}

В основе нашего построения лежит пример нигде не монотонной производной, не имеющей локальных экстремумов.
Известны примеры производных, имеющих плотное множество нулей, и принимающих значения разных знаков в любом интервале, см. \cite{Weil76} или \cite[Ch. 2, Theorem 6.6]{Bruckner}; простые явные примеры таких функций можно найти также в \cite[Corollary 4.2]{Kalyabin13}.
Ясно, что такие производные нигде не монотонны, но отсутствие у них локальных экстремумов не очевидно. В \cite[Ch. 6, Theorem 3.1]{Bruckner} была доказана следующая теорема:
\begin{theorem}
	Существует ограниченная аппроксимативно непрерывная производная, которая не имеет локальных экстремумов.
\end{theorem}
Приведённая в \cite{Bruckner} конструкция соответствующей производной весьма нетривиальна и в ней используется глубокая теорема Максимова \cite{Maximoff1940,Maximoff1943}.
%(согласно которой для всякой функции $g\colon [0,1]\to \R$ из первого класса Бэра, обладающей свойством Дарбу, существует гомеоморфизм $h\colon [0,1]\to [0,1]$ такой, что $g\circ h$ является производной).
Кроме того, не ясно, будет ли построенная в \cite{Bruckner} производная интегрируемой по Риману.
Например, в \cite{Stronska89} построены примеры ограниченных аппроксимативно непрерывных функций, множество точек разрыва которых имеет полную меру.

Наши основные результаты состоят в следующем:

\begin{theorem} \label{main_prop}
	Существует интегрируемая по Риману функция $f: [-1, 1] \rightarrow \R$, имеющая первообразную и удовлетворяющая одновременно двум условиям:
	\begin{enumerate}
		\item $f$ не имеет локальных экстремумов на $[-1, 1]$,
		\item $f$ не является монотонной ни на каком интервале $(a, b) \subset [-1, 1]$.
	\end{enumerate}
\end{theorem}

\begin{theorem} \label{inflection}
	Существует дифференцируемая функция $G$
	такая, что её производная $g$ имеет изолированный локальный экстремум, не являющийся точкой перегиба функции $G$. При этом $g$ интегрируема по Риману.
%	Существует дифференцируемая функция $G$ с интегрируемой по Риману производной $G'$ которая имеет изолированный локальный экстремум, не являющийся точкой перегиба функции $G$.
%	Существует дифференцируемая функция $G$ такая, что $x$ --- изолированная точка локального экстремума $G'$, $G'$ интегрируема по Риману, но $x$ не является точкой перегиба функции $G$.
\end{theorem}

\section{Обозначения}

Всюду в тексте под окрестностью точки $x$ подразумевается интервал $(x - \eps, x + \eps)$ для некоторого $\eps > 0$, пересечённый с областью определения рассматриваемой функции. Односторонняя окрестность --- полуинтервал $(x - \eps, x]$ или $[x, x + \eps)$.

Через $\mathcal{R}[a, b]$ будем обозначать множество функций, интегрируемых по Риману на отрезке $[a,b]$, а через $\mathcal{C}[a, b]$ --- непрерывных на $[a, b]$. 

\section{Вспомогательные результаты}

Приведём ряд общеизвестных теорем (см. например \cite{Redkozubov}).

\begin{theorem} (Критерий Лебега)
	Функция $f: [a, b] \to \R$ интегрируема по Риману на $[a, b]$ тогда и только тогда, когда она ограничена на $[a, b]$ и множество точек разрыва $f$ имеет меру нуль по Лебегу.
\end{theorem}

\begin{theorem} (Об интеграле с переменным верхним пределом) \label{variable_bound_integral}
	Пусть $f \in \mathcal{R}[a, b]$,\\ $F: [a, b] \to \R$ --- интеграл с переменным верхним пределом, т.е. $F(x) = \int_a^x f(t)dt$. Тогда
	\begin{itemize}
		\item $F \in \C[a, b]$
		\item Если $f$ непрерывна в точке $x \in [a, b]$, то $F$ дифференцируема в $x$ и $F'(x) = f(x)$.
	\end{itemize}
\end{theorem}

\begin{theorem} (О почленном дифференцировании функциональных рядов) \label{memberwise_differentiation}
	Пусть $I$ --- невырожденный промежуток, $\{F_n : I \to \R\}_{n=1}^\infty$ --- функциональная последовательность такая, что
	\begin{enumerate}
		\item $\forall n \in \N \ F_n$ дифференцируема на $I$;
		\item $\sum_{n=1}^\infty F_n$ сходится к функции $F$ хотя бы в одной точке;
		\item $\sum_{n=1}^\infty F_n'$ сходится равномерно к функции $f$ на $I$.
	\end{enumerate}
	Тогда $F$ дифференцируема на $I$ и $F' = f$.
\end{theorem}

\begin{theorem} (Признак Вейерштрасса)
	Пусть $\{a_n\}_{n=1}^\infty$ --- числовая последовательность, $\{f_n: X \subset\ \R \to \R\}_{n=1}^\infty$ --- функциональная последовательность и $\forall n \in \N, x \in X \ |f_n(x)| < a_n$. Тогда из сходимости ряда $\sum_{n=1}^\infty a_n$ следует равномерная сходимость ряда $\sum_{n=1}^\infty f_n$ на $X$.
\end{theorem}

\begin{theorem} (О непрерывности равномерно сходящегося ряда)
	Если функциональный ряд $\sum_{k=1}^\infty f_k$ сходится равномерно к функции $f$ на множестве $E$ и все $f_k$ непрерывны в точке $x \in E$, то и $f$ непрерывна в $x$.
\end{theorem}

\begin{theorem} (Об интегрируемости равномерно сходящегося ряда)
	Если функциональный ряд $\sum_{k=1}^\infty f_k$ сходится равномерно к функции $f$ на множестве $E$ и все $f_k$ интегрируемы на $[a, b]$, то и $f \in \mathcal{R}[a, b]$.
\end{theorem}

\begin{theorem} (Критерий выпуклости)
	Пусть $f$ дифференцируема на $(a, b)$. Тогда $f$ вогнута (выпукла) на $(a, b) \Longleftrightarrow f'$ возрастает (убывает) на $(a, b)$.
\end{theorem}

\section{Доказательство теоремы \ref{main_prop}}
		
%\begin{proof}[Доказательство теоремы \ref{main_prop}]
	 Построим последовательность функций $\{f_k \colon [-1, 1] \to \R\}_{k=1}^\infty$. Искомая функция $f$ будет представлять собой ряд:
	 \[
	 	f(x) = \sum_{k=1}^\infty \frac{f_k(x)}{2^k}.
	 \]
	
	Чтобы определить $f_1$, воспользуемся тем, что любое число $x \in [0, 1)$ лежит в единственном полуинтервале $\left[1 - \frac{1}{n}, 1 - \frac{1}{n+1}\right)$ для некоторого $n \in \N$, а значит, однозначно записывается в виде $(1-t)\ps{1 - \frac{1}{n}} + t\ps{1 - \frac{1}{n+1}}, \ n \in \N, \ t \in [0, 1)$. Используя это представление,
	 \[
	 	f_1(x) = \begin{cases}
	 		(-1)^n(1 - 2t), & x \in \left[\frac{1}{2}, 1\right); \\
	 		t, & x \in \left[0, \frac{1}{2}\right); \\
	 		0, & x = 1; \\
	 		-f_1(-x), & x \in [-1, 0).
	 	\end{cases}
	 \]
	 
	 По сути функция $f_1$ представляет собой что-то похожее на $\sin(\frac{\pi}{x^2-1}) \cdot \sign(x)$, но является линейной на каждом промежутке монотонности, что будет доказано ниже. Заметим, что область значений $f_1$ та же, что и область определения --- отрезок $[-1, 1]$. С учётом этого индуктивное построение функций $f_k$ при $k > 1$ осуществляется очень просто:
	 \[
	 	f_k(x) = f_1(f_{k-1}(x)).
	 \]
	 Далее будет также полезно более общее определение, которое тривиально проверяется по индукции:
	 \[
	 	\forall l < k \ \ f_k(x) = f_l(f_{k-l}(x)).
	 \]
	
	\begin{figure}[H]
	\begin{center}
		\begin{tikzpicture}[scale=4]
		\begin{scope}[xscale=3.5]
		\tikzmath{
			\K=9;
			\M=\K+1;
		}
		\draw[very thin] (1/\M,0) -- (1-1/\M,0);
		\draw[very thin,-latex] (1,0) -- (1.1,0) node[above] {$x$};
		\draw[very thin,-latex] (0,-1.2) -- (0,1.2) node[right] {$y$};
		\node[below right] at (0,1) {$1$};
		\node[below right] at (0,0) {$-1$};
		\node[above right] at (0,-1) {$-1$};
		\node[above left] at (1/2,0) {$0$};
		\node[below right] at (1,0) {$1$};
		\draw[very thin] (1,-1) -- (1,1);
		\node[scale=2] at (0,0) {.};
		\node[scale=2] at (1,0) {.};
		\node[scale=1] at (1/\M/2,1/2) {...};
		\draw[dotted] (0,0) -- (1/\M, 0);
		\node[scale=1] at (1/\M/2,-1/2) {...};
		\node[scale=1] at (1-1/\M/2,1/2) {...};
		\draw[dotted] (1-1/\M,0) -- (1, 0);
		\node[scale=1] at (1-1/\M/2,-1/2) {...};
		\draw[very thin, gray] (0, 1) -- (1, 1);
		\draw[very thin, gray] (0, -1) -- (1, -1);
		\draw (1/3,-1) -- (2/3,1);
		\foreach \n [
			evaluate=\n as \s using (-1)^(\n)
		] in {3,...,\K} {
			\draw (1/\n,\s) -- ({1/(\n+1)}, -\s);
			\node[color=blue,scale=2] at (1/\n,0) {.};
			\draw ({1-1/\n},-\s) -- ({1-1/(\n+1)}, \s);
			\node[color=blue,scale=2] at (1-1/\n,0) {.};
			\ifthenelse{\n<7}{%
				\draw[very thin, gray, dashed] (1/\n,0) -- (1/\n, \s);
				\draw[very thin, gray, dashed] (1-1/\n,0) -- (1-1/\n, -\s);
			}{}
			\ifthenelse{\n=3}{%
				\node[above] at (1/\n,0) {$\frac{1-n}{n}$};
				\node[below] at ({1/(\n+1)},0) {$\frac{-n}{n+1}$};
			}{}
			\ifthenelse{\n=3}{%
				\node[above] at ({(\n-1)/\n},0) {$\frac{n-1}{n}$};
				\node[below] at ({\n/(\n+1)},0) {$\frac{n}{n+1}$};
			}{}
		}
		\node[align=center,black] at (1/2,-9/8) {$f_1$};
		\end{scope}
	\end{tikzpicture}
		\begin{tikzpicture}[scale=4]
		\begin{scope}[xscale=3.5]
		\tikzmath{
			\K=9;
			\M=\K+1;
			\KK = 7;
			\MM = \KK + 2;
		}
		\draw[very thin] (1/\M,0) -- (1/4,0);
		\draw[very thin] (3/4,0) -- (1-1/\M,0);
		\draw[very thin] (1/4 + 1/60,0) -- (1/3 - 1/60,0);
		\draw[very thin] (2/3 + 1/60,0) -- (3/4 - 1/60,0);
		\draw[very thin] ({1/3 + 1/(3*\MM)},0) -- ({2/3-1/(3*\MM)},0);
		\draw[very thin,-latex] (1,0) -- (1.1,0) node[above] {$x$};
		\draw[very thin,-latex] (0,-1.2) -- (0,1.2) node[right] {$y$};
		\node[below right] at (0,1) {$1$};
		\node[below right] at (0,0) {$-1$};
		\node[above right] at (0,-1) {$-1$};
		\node[above left] at (1/2,0) {$0$};
		\node[below right] at (1,0) {$1$};
		\draw[very thin] (1,-1) -- (1,1);
		\node[scale=2] at (0,0) {.};
		\node[scale=2] at (1,0) {.};
		\node[scale=1] at (1/\M/2,1/2) {...};
		\draw[dotted] (0,0) -- (1/\M, 0);
		\node[scale=1] at (1/\M/2,-1/2) {...};
		\node[scale=1] at (1-1/\M/2,1/2) {...};
		\draw[dotted] (1-1/\M,0) -- (1, 0);
		\node[scale=1] at (1-1/\M/2,-1/2) {...};
		\draw[very thin, gray] (0, 1) -- (1, 1);
		\draw[very thin, gray] (0, -1) -- (1, -1);

		\node[scale=1] at ({1/3 + 1/(3*2*\KK)},1/2) {...};
		\draw[dotted] (1/3 - 1/60, 0) -- (1/3 + 1/\MM, 0);
		\node[scale=1] at ({1/3 + 1/(3*2*\KK)},-1/2) {...};
		\node[scale=1] at ({2/3 - 1/(3*2*\KK)},1/2) {...};
		\draw[dotted] (2/3-1/\MM,0) -- (2/3 + 1/60, 0);
		\node[scale=1] at ({2/3 - 1/(3*2*\KK)},-1/2) {...};

		\draw[dotted] (1/4+1/60,0) -- (1/4, 0);
		\draw[dotted] (3/4-1/60,0) -- (3/4, 0);

		\draw[very thin, gray, dashed] (1/4,-1) -- (1/4, -1/6);
		\draw[very thin, gray, dashed] (1/4,0) -- (1/4, 1);
		\draw[very thin, gray, dashed] (3/4,-1) -- (3/4, -1/6);
		\draw[very thin, gray, dashed] (3/4,0) -- (3/4, 1);
		\draw[very thin, gray, dashed] (1/3,-1) -- (1/3, 0);
		\draw[very thin, gray, dashed] (1/3,1/6) -- (1/3, 1);
		\draw[very thin, gray, dashed] (2/3,-1) -- (2/3, 0);
		\draw[very thin, gray, dashed] (2/3,1/6) -- (2/3, 1);

		\draw (4/9,-1) -- (5/9,1);
		\foreach \n [
			evaluate=\n as \s using (-1)^(\n)
		] in {3,...,\KK} {
			\draw ({1/3 + 1/(3*\n)},\s) -- ({1/3 + 1/(3*(\n+1))}, -\s);
			\node[color=blue,scale=2] at (1/\n,0) {.};
			\draw ({2/3-1/(3*\n)},-\s) -- ({2/3-1/(3*(\n+1))}, \s);
			\node[color=blue,scale=2] at (1-1/\n,0) {.};
			\ifthenelse{\n=3}{%
				\node[above] at (1/\n,0) {$\frac{1-n}{n}$};
				\node[below] at ({1/(\n+1)},0) {$\frac{-n}{n+1}$};
			}{}
			\ifthenelse{\n=3}{%
				\node[above] at ({(\n-1)/\n},0) {$\frac{n-1}{n}$};
				\node[below] at ({\n/(\n+1)},0) {$\frac{n}{n+1}$};
			}{}
		}

		\draw (10/36,1) -- (11/36,-1);
		\draw (25/36,1) -- (26/36,-1);
		\foreach \n [
			evaluate=\n as \s using (-1)^(\n+1)
		] in {3,...,4} {
			\draw ({1/4 + 1/(12*\n)},\s) -- ({1/4 + 1/(12*(\n+1))}, -\s);
			\draw ({1/3-1/(12*\n)},-\s) -- ({1/3-1/(12*(\n+1))}, \s);
			\draw ({2/3 + 1/(12*\n)},\s) -- ({2/3 + 1/(12*(\n+1))}, -\s);
			\draw ({3/4-1/(12*\n)},-\s) -- ({3/4-1/(12*(\n+1))}, \s);
		}
		\node[align=center,black] at (1/2,-9/8) {$f_2$};
		\end{scope}
	\end{tikzpicture}
	\end{center}
	\caption{Функции $f_1$ и $f_2$.}
	\end{figure}

	Всё дальнейшее изложение будет существенно опираться на следующую лемму, описывающую структуру $f_k$:
	
	\begin{lemma} \label{f_k_properties}
		Для любого $k \in \N$ отрезок $[-1, 1]$ представим как $E_k \cup \bigcup_{i=1}^\infty A_{k,i}$, где $E_k$ --- множество точек разрыва $f_k$, а $A_{k,i}$ --- невырожденные отрезки $[a_{k,i}, b_{k,i}]$, на каждом из которых $f_k$ --- линейная функция. При этом верны следующие свойства:
		\begin{enumerate}
			
			\item $f_k(E_k) = 0$, и если $x \in E_k$, то $\forall \eps > 0$ $f_k(U_{+\eps(x)}) = [-1, 1]$ при $x \neq 1$ и $f_k(U_{-\eps}(x)) = [-1, 1]$ при $x \neq -1$;
			
			\item При всех $i$ $f_k(A_{k,i}) = [-1, 1]$;
			
			\item
			\[
				E_1 = \{-1, 1\}, \ E_{k+1} = E_k \cup \bigcup_{i=1}^\infty \{a_{k,i}, b_{k,i}\};
			\]
			
			\item
			\begin{gather*}
				\{A_{1,s} \such s \in \N\} = \set{\sbr{-\frac{1}{2}, \frac{1}{2}}} \cup \bigcup_{n=2}^\infty \set{\sbr{\frac{1}{n+1} - 1, \frac{1}{n} - 1}, \sbr{1 - \frac{1}{n}, 1 - \frac{1}{n+1}}}, \\
				\{A_{{k+1},s} \such s \in \N\} = \{h_{k,i}(A_{1,j}) \such i, j \in \N \},
			\end{gather*}
			где $h_{k,i}(x) = (b_{k,i} - a_{k,i}) \cdot \frac{x + 1}{2} + a_{k,i}$.
			
			\item $\{a_{k,i} \such i \in \N\} = \{b_{k,i} \such i \in \N\}$
		\end{enumerate}
	\end{lemma}
	\begin{proof}
		По индукции.
		
		Пусть $x \in \left[\frac{1}{2}, 1\right)$, тогда рассмотрим полуинтервал $\left[1 - \frac{1}{n}, 1 - \frac{1}{n+1}\right)$, в который попадает $x$. На этом полуинтервале $f_1$ задана линейно, причём $f_1\ps{1 - \frac{1}{n}} = (-1)^{n+1}$ и
		\[
			\lim_{x \to 1 - \frac{1}{n+1}} f_1(x) = (-1)^n = (-1)^{n+1} \cdot (-1) = f_1\ps{1 - \frac{1}{n+1}}.
		\]
		Это значит, что $f_1$ линейна на всём отрезке $\sbr{1 - \frac{1}{n}, 1 - \frac{1}{n+1}}$ и $f_1\ps{\sbr{1 - \frac{1}{n}, 1 - \frac{1}{n+1}}} = [-1, 1]$.
		Отрезок $\sbr{-\frac{1}{2}, \frac{1}{2}}$ тоже удовлетворяет этому свойству, поскольку $f_1$ линейна на $\lsi{0, \frac{1}{2}}$, \\ $\lim_{x \to \frac{1}{2}} f(x) = f\ps{\frac{1}{2}} = 1$, $f_1(0) = 0$ и $f_1$ нечётна. Пусть теперь $x = 1$. Поскольку \\ $\lim_{n \to \infty} \ps{1 - \frac{1}{n}} = 1$, любая окрестность точки $1$ включает в себя отрезки $\sbr{1 - \frac{1}{n}, 1 - \frac{1}{n+1}}$ для достаточно больших $n$, а на них $f_1$ принимает все значения от $-1$ до $1$. Это создаёт разрыв второго рода в точке $1$, и по определению $f_1(1) = 0$. В силу нечётности $f_1$ всё сказанное верно и при $x < -\frac{1}{2}$. Пятый пункт для $f_1$ следует из того, что у каждого отрезка $A_{1,s}, s \in \N$ есть непосредственно следующий и предыдущий. База доказана.
		
		Докажем переход от $k$ к $k + 1$. Во-первых, при всех $i$ и $j$ $h_{k,i}(A_{1,j})$ действительно отрезок, потому что это образ отрезка под действием линейной функции. Во-вторых, поймём, почему множества $E_{k+1}$, $\set{A_{k+1,s} \such s \in \N}$, полученные по утверждаемой формуле, покрывают вместе отрезок $[-1, 1]$. Воспользуемся предположением, что это верно для $k$, и тем, что объединение образов равно образу объединения:
		\[
			\bigcup_{j=1}^\infty h_{k,i}(A_{1,j}) = h_{k,i}( \, \bigcup_{j=1}^\infty A_{1,j}) = h_{k,i}((-1, 1)) = (a_{k,i}, b_{k,i}).
		\]
		Применяя это равенство для каждого $i$,
		\[
			E_{k+1} \cup \bigcup_{s=1}^\infty A_{k+1,s} = E_k \cup \bigcup_{i=1}^\infty \{a_{k,i}, b_{k,i}\} \cup \bigcup_{i=1}^\infty \bigcup_{j=1}^\infty h_{k,i}(A_{1,j}) = E_k \cup \bigcup_{i=1}^\infty A_{k,i} = [-1, 1].
		\]
		Множества $E_k$, $\bigcup_{i=1}^\infty \{a_{k,i}, b_{k,i}\}$ = $\{a_{k,i} \such i \in \N\} = \{b_{k,i} \such i \in \N\}$ и $\bigcup_{i=1}^\infty (a_{k,i}, b_{k,i})$ рассмотрим по отдельности, поскольку на них поведение $f_{k+1}$ принципиально различается.
		\begin{itemize}
			\item $x \in E_k$. По предположению индукции $\forall \eps > 0 \ f_k(U_{\pm\eps}(x)) = [-1, 1]$ (в точках $x = \pm 1$ рассматривается только одна окрестность) и $f_k(x) = 0$. Но тогда $\forall \eps > 0 \ f_{k+1}(U_\pm\eps(x)) = f_1(f_k(U_\pm\eps(x_0))) = f_1([-1, 1]) = [-1, 1]$ и $f_{k+1}(x) = f_1(f_k(x)) = f_1(0) = 0$. То есть $x$ --- точка разрыва $f_{k+1}$, для которой верно первое свойство.
			
			\item $x = a_{k,i} = b_{k,j}$ для некоторых $i, j$. Это значит, что $f_k(x) = \pm 1$ и $f_k$ --- линейная и непостоянная функция в любых достаточно малых односторонних окрестностях $U_{+\eps}(x) \subset A_{k,i}$ и $U_{-\eps}(x) \subset A_{k,j}$. Отсюда $f_k(U_{+\eps}(x)) = U_{\alpha \eps}(f_k(x)) = U_{\alpha \eps}(\pm 1)$ и $f_k(U_{-\eps}(x)) = U_{\beta \eps}(f_k(x)) = U_{\beta \eps}(\pm 1)$, где $\alpha, \beta \neq 0$. Следовательно, $f_{k+1}(U_{+\eps}(x)) = f_1(U_{\alpha \eps}(\pm 1)) = [-1, 1]$ и $f_{k+1}(U_{-\eps}(x)) = f_1(U_{\beta \eps}(\pm 1)) = [-1, 1]$. Это приводит к разрыву второго рода $f_{k+1}$ в точке $x$. При этом $f_{k+1}(x) = f_1(f_k(x)) = f_1(\pm 1) = 0$, и свойство 1 снова выполнено.
			
			\item $x \in (a_{k,i}, b_{k,i})$. Тогда по доказанному $x \in h_{k,i}(A_{1,j}) \subset A_{k,i}$ для некоторых $i, j$. Если показать, что $f_{k+1}$ линейна на $h_{k,i}(A_{1,j})$, то $f_{k+1}$ окажется непрерывной в точке $x$. Воспользуемся формулой $f_{k+1} = f_1 \circ f_k$. В предположении, что $f_k$ линейна на $A_{k,i}$ и $f_k(A_{k,i}) = [-1, 1]$, можно почти однозначно записать формулу для $f_k$ на этом отрезке:
			\[
				f_k(x) = \pm (2t - 1), \ t = \frac{x - a_{k,i}}{b_{k,i} - a_{k,i}}.
			\]
			Зная это, можно вычислить $f_k(h_{k,i}(A_{1,j}))$: граница образа будет образом границы $A_{1,j}$. Посмотрим, во что переходят точки $a_{1,j}$ и $b_{1,j}$: 
			\begin{gather*}
				f_k(h_{k,i}(a_{1,j})) = \pm \ps{2\frac{h_{k,i}(a_{1,j}) - a_{k,i}}{b_{k,i} - a_{k,i}} - 1} = \pm \ps{2\frac{a_{1,j} + 1}{2} - 1} = \pm a_{1,j}, \\
				f_k(h_{k,i}(b_{1,j})) = \pm \ps{2\frac{h_{k,i}(b_{1,j}) - a_{k,i}}{b_{k,i} - a_{k,i}} - 1} = \pm \ps{2\frac{b_{1,j} + 1}{2} - 1} = \pm b_{1,j}.
			\end{gather*}
			Стало быть, $f_k(h_{k,i}(A_{1,j})) = \pm [a_{1,j}, b_{1,j}] = \pm A_{1,j}$, причём из 4-го пункта леммы $-A_{1,j}$ также является $A_{1,l}$ для некоторого $l$. Значит, $f_1$ линейна на $f_k(h_{k,i}(A_{1,j}))$, а $f_{k+1} = f_1 \circ f_k$ линейна на $h_i(A_{1,j})$, и $f_{k+1}(h_{k,i}(A_{1,j})) = f_1(\pm A_{1,j}) = [-1, 1]$.
		\end{itemize}
	Как показывает этот разбор случаев, $E_{k+1}$ действительно является множеством точек разрыва $f_{k+1}$, а на каждом отрезке $A_{k+1,s}$ $f_{k+1}$ линейна и принимает значения от $-1$ до $1$.
	
	Докажем последний, пятый пункт, используя рекуррентную формулу для $A_{k+1,s}$ и тот же пункт для $f_1$. При каждом $i$ имеем:
	\[
		\bigcup_{j=1}^\infty h_{i,k}(a_{1,j}) = h_{i,k}(\, \bigcup_{j=1}^\infty a_{1,j}) = h_{i,k}(\, \bigcup_{j=1}^\infty b_{1,j}) = \bigcup_{j=1}^\infty h_i(b_{1,j}),
	\]
	поэтому
	\[
		\bigcup_{s=1}^\infty \{a_{k+1,s}\} = \bigcup_{i=1}^\infty \bigcup_{j=1}^\infty h_{i,k}(a_{1,j}) = \bigcup_{i=1}^\infty \bigcup_{j=1}^\infty h_{i,k}(b_{1,j}) = \bigcup_{s=1}^\infty \{b_{k+1,s}\}.
	\]
	
	\end{proof}

	\begin{corollary} \label{integration}
		При всех $k \in \N$ $f_k$ интегрируема по Риману на $[-1, 1]$ и $\int_{-1}^1 f_k(t) dt = 0$.
	\end{corollary}

	\begin{proof}
		Из третьего пункта леммы тривиально следует, что $E_k$ не более чем счётно для любого $k$, и в силу критерия Лебега $f_k$ интегрируема по Риману на $[-1, 1]$. Кроме того, $f_k$ --- нечётная функция при всех $k$, что также непосредственно проверяется по индукции: $f_k(-x) = f_1(f_{k-1}(-x)) = f_1(-f_{k-1}(x)) = -f_1(f_{k-1}(x)) = -f_k(x)$. Интеграл от нечётной функции по отрезку $[-1, 1]$ равен нулю. Следствие доказано.
		
	\end{proof}

	Теперь мы готовы доказать заявленные свойства функции $f$. Начнём с того, что все $f_k$ ограничены по модулю единицей, откуда по теореме Вейерштрасса следует, что ряд $\sum_{k = 1}^\infty \frac{f_k(x)}{2^k}$ сходится равномерно. Тогда по предыдущему следствию и теореме об интегрируемости равномерно сходящегося ряда $f$ интегрируема по Риману на $[-1, 1]$.
	
	Чтобы установить наличие первообразной у $f$ на $[-1, 1]$, покажем, что её имеет $f_k$ при всех $k$, а следовательно, и $\frac{f_k}{2^k}$. Прибавлением нужной константы можно добиться того, что ряд из первообразных будет сходиться хотя бы в одной точке. Этот ряд будет удовлетворять всем условиям теоремы о почленном дифференцировании, поэтому его сумма будет первообразной функции $f$. Таким образом, достаточно доказать следующий факт:
	
	\begin{proposition}
		Функция $F_k(x) := \int_{-1}^x f_k(t) dt$ является первообразной функции $f_k$ для всех $k \in \N$.
	\end{proposition}

	\begin{proof}
		Теорема об интеграле с переменным верхним пределом гласит, что $F_k'(x) = f_k(x)$ во всех точках $x$, в которых $f_k$ непрерывна. В остальных точках докажем равенство по определению, то есть установим наличие предела:
		\begin{align*}
			\lim_{x \to x_0} \frac{F_k(x) - F_k(x_0)}{x - x_0} = \lim_{x \to x_0} \frac{\int_{-1}^x f_k(t) dt - \int_{-1}^{x_0} f_k(t) dt}{x - x_0} = \lim_{x \to x_0} \frac{\int_{x_0}^x f_k(t) dt}{x - x_0} = f_k(x_0),
		\end{align*}
		где $x_0 \in E_k$. При этом $f_k(x_0) = 0$ по лемме \ref{f_k_properties}.
		
		Рассмотрим интеграл $\int_{x_0}^x f_k(t) dt$. Сначала сведём его к случаю $x_0 = 1$ или $-1$ (что неважно, так как $f_k$ нечётна). Поскольку $x_0 \in E_k$, можно выбрать минимальное $l \leq k$ такое, что $x_0 \in E_l$. Если $l = 1$, то уже $x_0 = \pm 1$. Если $l > 1$, то в силу выбора $l$ и пункта 5 леммы \ref{f_k_properties} $x_0 = a_{l-1,i} = b_{l-1, j}$ для некоторых $i$ и $j$. Другими словами, $f_{l-1}$ представляет собой линейную функцию по отдельности слева и справа от точки $x_0$, причём непостоянную. Это позволяет сделать замену переменной в интеграле для любого достаточно малого $x$:
		\[
			\int_{x_0}^x f_k(t) dt = \int_{x_0}^x f_{k-l+1}(f_{l-1}(t)) dt = \int_{x_0}^x f_{k-l+1}(at + b) dt = \frac{1}{a} \int_{f_{l-1}(x_0)}^{f_{l-1}(x)} f_{k-l+1}(s) ds.
		\]
		Здесь уже $f_{l-1}(x_0) = \pm 1$, а в силу нечётности $f_{k-l+1}$ достаточно рассмотреть только случай $f_{l-1}(x_0) = -1$. Выразим теперь приращение аргумента через $f_{l-1}$:
		\[
			x - x_0 = f_{l-1}^{-1}(f_{l-1}(x)) - f_{l-1}^{-1}(f_{l-1}(x_0)) = \frac{f_{l-1}(x) - b}{a} - \frac{f_{l-1}(x_0) - b}{a} = \frac{1}{a}(f_{l-1}(x) - f_{l-1}(x_0)).
		\]
		Так как $f_{l-1}(x) \to f_{l-1}(x_0)$ при $x \to x_0$, предел принимает следующий вид:
		\[
			\lim_{x \to x_0} \frac{\int_{x_0}^x f_k(t) dt}{x - x_0} = \lim_{x \to x_0} \frac{\int_{f_{l-1}(x_0)}^{f_{l-1}(x)} f_{k-l+1}(s) ds}{f_{l-1}(x) - f_{l-1}(x_0)} = \lim_{y \to -1} \frac{\int_{-1}^y f_{k-l+1}(s) ds}{y + 1}.
		\]
		Остаётся доказать случай $x_0 = -1$.
		
		Разобьём рассматриваемый интеграл на две части:
		\[
			\int_{-1}^x f_k(t) dt = \int_{-1}^{\frac{1}{n+1} - 1} f_k(t) dt + \int_{\frac{1}{n+1} - 1}^x f_k(t) dt
		\]
		где $x \in \left(\frac{1}{n+1} - 1, \frac{1}{n} - 1\right]$.
		Первое слагаемое можно понимать как несобственный интеграл и записать его следующим образом:
		\[
			\int_{-1}^{\frac{1}{n+1} - 1} f_k(t) dt = \lim_{N \to \infty} \int_{\frac{1}{N+1} - 1}^{\frac{1}{n+1} - 1} f_k(t) dt = \lim_{N \to \infty} \sum_{m = n+1}^{N} \int_{\frac{1}{m+1} - 1}^{\frac{1}{m} - 1} f_k(t) dt = \sum_{m = n+1}^\infty \int_{\frac{1}{m+1} - 1}^{\frac{1}{m} - 1} f_k(t) dt.
		\]
		Заметим, что $f_k(t) = f_{k-1}(f_1(t))$, а $f_1$ --- линейная функция на каждом отрезке вида $\sbr{\frac{1}{m+1} - 1, \frac{1}{m} - 1}$, переводящая его в $[-1, 1]$. Значит, по следствию \ref{integration}
		\[
			\int_{\frac{1}{m+1} - 1}^{\frac{1}{m} - 1} f_{k-1}(f_1(t)) dt = С \int_{-1}^1 f_{k-1}(s) ds = 0.
		\]
		Остаётся только второе слагаемое, и вся дробь уже оценивается непосредственно:
		\[
			\md{\frac{\int_{-1}^x f_k(t) dt}{x + 1}} < \frac{\int_{\frac{1}{n+1} - 1}^{\frac{1}{n} - 1} |f_k(t)| dt}{\frac{1}{n+1}} \leq (n+1) \ps{\frac{1}{n} - \frac{1}{n+1}} = \frac{1}{n}.
		\]
		Поскольку $\frac{1}{n} \to 0$ при $x \to -1$,
		\[
			\lim_{x \to -1} \frac{\int_{-1}^x f_k(t) dt}{x + 1} = 0.
		\]
		А значит, с учётом всего сказанного $F_k'(x) = f_k(x)$ при всех $x \in [-1, 1]$.
		
	\end{proof}
	
	Чтобы перейти к следующим свойствам функции $f$, нам понадобится ещё одно простое наблюдение:
	
	\begin{corollary} (из леммы \ref{f_k_properties}) \label{interval_size}
		При всех $k, i$ $|A_{k,i}| \leq \frac{1}{2^{k-1}}$ и $\forall l < k \ \exists j: A_{k,i} \subset A_{l,j}$.
	\end{corollary}

	\begin{proof}
		Проверим по индукции, используя 4-ю часть леммы. 	При $k=1$ имеется отрезок $A_{1,i} = \sbr{-\frac{1}{2}, \frac{1}{2}}$, его длина равна $\frac{1}{2^{1-1}}$. При других $i$ длина отрезка $A_{1,i}$ равна $\ps{1 - \frac{1}{n+1}} - \ps{1 - \frac{1}{n}}$, что меньше единицы. Для шага индукции воспользуемся непосредственно формулой для $A_{k+1,s}$:
		\[
			|A_{k+1,s}| = |h_{k,i}(A_{1,j})| = \md{(b_{k,i} - a_{k,i}) \cdot \frac{A_{1,j} + 1}{2} + a_{k,i}} = |A_{k,i}| \cdot \frac{|A_{1,j}|}{2} \leq \frac{1}{2^k}.
		\]
		Для тех же $s, i$ $A_{k+1,s} \subset A_{k,i}$, а по предположению $\forall l < k \ \exists j : A_{k,i} \subset A_{l,j}$.
		
	\end{proof}

	\begin{proposition} \label{f_discontinuities}
		$\forall x_0 \in \bigcup_{k=1}^\infty E_k \ \varliminf_{x \to x_0} f(x) < f(x_0) < \varlimsup_{x \to x_0} f(x)$.
	\end{proposition}

	Это утверждение показывает, что множество точек разрыва $f$ содержит $\bigcup_{k=1}^\infty E_k$. По теореме о непрерывности равномерно сходящегося ряда верно и обратное включение, но нам оно не понадобится.

	\begin{proof}
		Пусть $x_0 \in \bigcup_{k=1}^\infty E_k$. Выберем минимальное $k$, при котором $x_0 \in E_k$. Тогда $S_{k-1}$ --- $(k-1)$-я частичная сумма $f$ --- непрерывна в $x_0$ (в случае $k = 1$ $S_{k-1} \equiv 0$), а $f_k$ по лемме \ref{f_k_properties} принимает значения от -1 до $1$ в любой окрестности точки $x_0$. Возьмём достаточно малую для наших нужд окрестность:
		\[
			\exists \delta_0 > 0 \ \forall x \in U_{\delta_0}(x_0) \ |S_{k-1}(x) - S_{k-1}(x_0)| < \frac{1}{2^{k+1}}.
		\]
		Теперь для любого заданного $\delta \in (0, \delta_0)$ можно найти точки в $U_\delta(x_0)$, в которых $f_k$ достигает значений $\pm 1$:
		\[
			\exists x_1, x_2 \in U_\delta(x_0) : f_k(x_1) = 1, \ f_k(x_2) = -1.
		\]
		Эти точки являются концами некоторых отрезков $A_{k,i}$, $A_{k,j}$, поэтому по пунктам 1 и 3 леммы \ref{f_k_properties} $\forall l > k$ $f_l(x_0) = f_l(x_1) = f_l(x_2) = 0$, а также $f_k(x_0) = 0$. Отсюда неравенство:
		\[
			f(x_1) = S_k(x_1) = S_{k-1}(x_1) + \frac{f_k(x_1)}{2^k} > \ps{S_{k-1}(x_0) - \frac{1}{2^{k+1}}} + \frac{1}{2^k} = f(x_0) + \frac{1}{2^{k+1}},
		\]
		и аналогично
		\[
			f(x_2) < \ps{S_{k-1}(x_0) + \frac{1}{2^{k+1}}} - \frac{1}{2^k} = f(x_0) - \frac{1}{2^{k+1}}.
		\]
		Но это означает, что
		\[
			\varliminf_{x \to x_0} f(x) \leq f(x_0) - \frac{1}{2^{k+1}} < f(x_0) < f(x_0) + \frac{1}{2^{k+1}} \leq \varlimsup_{x \to x_0} f(x).
		\]
		
	\end{proof}
	
	\begin{proposition} \label{no_extrema}
		$f$ не имеет локальных экстремумов ни в какой точке $x_0 \in [-1, 1]$.
	\end{proposition}

	\begin{proof}
		Если $x_0 \in E_k$ для некоторого $k$, то ввиду предыдущего утверждения $x_0$ не может являться локальным экстремумом. Рассмотрим оставшийся случай.
		
		Зафиксируем любую окрестность $U_\delta(x_0)$. Сразу заметим, что $x_0 \neq \pm 1$, поэтому эта окрестность двусторонняя и $x_0$ --- её внутренняя точка. Раз $x_0 \notin E_k$ ни для какого $k$, то $x_0 \in (a_{k,i}, b_{k,i})$ для всех $k$ ($i$ подразумевается зависящим от $k$). Тогда по следствию \ref{interval_size} для достаточно большого $k$ $\exists i: x_0 \in (a_{k,i}, b_{k,i}) \subset [a_{k,i}, b_{k,i}] \subset U_\delta(x_0)$. По тому же следствию $\forall l < k \ \exists j : A_{k,i} \subset A_{l,j}$. Это говорит о том, что $S_k$ есть линейная функция на $A_{k,i}$, но, возможно, постоянная. Если $S_k$ оказалась постоянной, то увеличим $k$ на единицу --- $S_{k+1}$ таковой уже не будет, и по-прежнему $x_0 \in (a_{k+1,j}, b_{k+1,j}) \subset (a_{k,i}, b_{k,i}) \subset U_\delta(x_0)$. Это даёт право считать, что $S_k$ строго монотонна на $A_{k,i}$.
		
		По лемме \ref{f_k_properties} $a_{k,i}, b_{k,i} \in E_{k+1}$ и $f_{k+1}([a_{k,i}, x_0)) = f_{k+1}((x_0, b_{k,i}]) = [-1, 1]$. В частности,
		\begin{gather*}
			\exists x_1 \in [a_{x,i}, x_0) \colon f_{k+1}(x_1) = f_{k+1}(x_0), \\
			\exists x_2 \in (x_0, b_{x,i}] \colon 	f_{k+1}(x_2) = f_{k+1}(x_0).
		\end{gather*}
		Но то же равенство верно и при любом $l > k + 1$, поскольку 
		\[
			f_l(x_i) = f_{l-k-1}(f_{k+1}(x_i)) = f_{l-k-1}(f_{k+1}(x_0)) = f_l(x_0), \ i = 1, 2.
		\]
		Точки $x_1$ и $x_2$ свидетельствуют о том, что $x_0$ не является точкой экстремума в $U_\delta(x_0)$:
		\begin{gather*}
			(f(x_1) - f(x_0))(f(x_2) - f(x_0)) = (S_k(x_1) - S_k(x_0))(S_k(x_2) - S_k(x_0)) < 0.
		\end{gather*}
		Поскольку это верно в любой окрестности $U_\delta(x_0)$, утверждение доказано.
	
	\end{proof}

	Осталось последнее:
	\begin{proposition}
		$f$ нигде не монотонна на $[-1, 1]$.
	\end{proposition}

	\begin{proof}
		Предположим противное: найдётся интервал $(a, b) \subset [-1, 1]$, на котором $f$ монотонна. С одной стороны, монотонная функция может иметь точки разрыва только первого рода. С другой стороны, вспомним, что $f$ является производной, а у производной не бывает точек разрыва первого рода. Значит, $f$ должна быть непрерывной на $(a, b)$.
		
		Теперь возьмём такое $k$, что $\frac{1}{2^{k-1}} < b - a$. Тогда по следствию \ref{interval_size} $(a, b)$ пересекается либо с $\bigcup_{i=1}^\infty \{a_{k,i}, b_{k,i}\}$, либо с $E_k$, то есть пересекается с $E_{k+1}$. По утверждению \ref{f_discontinuities} это означает, что $(a, b)$ содержит точку разрыва $f$. Противоречие.
		
	\end{proof}

	Таким образом, построенный пример доказывает теорему \ref{main_prop}.
%\end{proof}

%\begin{note}
%	Ещё раз заметим два дополнительных свойства, которыми обладает полученная производная: интегрируемость по Риману и ограниченность. Последнее нам понадобится, чтобы доказать основное следствие.
%\end{note}

\section{Доказательство теоремы \ref{inflection}}

Рассмотрим функцию $g(x) = f(x) \cdot \sign(x)$, где $f$ --- функция из предыдущего примера. Первообразной функции $g$ на $[-1, 1]$ будет $G(x) = F(x) \cdot \sign(x)$, где $F$ --- первообразная функции $f$ такая, что $F(0) = 0$. Утверждение состоит в том, что $0$ является точкой строгого правостороннего локального минимума функции $f$. Если это так, то $G$ будет удовлетворять всем требованиям теоремы \ref{inflection}:
\begin{itemize}
	\item $g$, как и $f$, интегрируема по Риману;
	\item При $x \geq 0$ $g(x) = f(x)$, а при $x \leq 0$ $g(x) = -f(x) = f(-x)$, поэтому 0 --- единственная точка двустороннего локального экстремума функции $g$;
	\item если $f$ нигде не монотонна, то и $g$ тоже. Тогда по критерию выпуклости $G$ не является ни выпуклой, ни вогнутой ни на каком интервале и, следовательно, не имеет точек перегиба.
\end{itemize}

Осталось доказать, что $f(x) > f(0)$ в некоторой правосторонней окрестности $f$. Для этого проанализируем соответствующие интервалы линейности функций $f_k$:
\begin{proposition} \label{middle_intervals}
	При всех $k$ существует отрезок $A_{k,i}$ из леммы \ref{f_k_properties}, равный $\sbr{-\frac{1}{2^k}, \frac{1}{2^k}}$, и $f_k$ возрастает на $A_{k,i}$.
\end{proposition}

\begin{proof}
	Ещё раз прибегнем к индукции, основываясь на 4-й части леммы \ref{f_k_properties}. База уже доказана. Переход: если $A_{k,i} = \sbr{-\frac{1}{2^k}, \frac{1}{2^k}}$, $A_{1,j} = \sbr{-\frac{1}{2}, \frac{1}{2}}$, то
	\[
		A_{k+1,s} = h_{k,i}(A_1,j) = \frac{1}{2^{k-1}} \cdot \frac{A_{1,j} + 1}{2} - \frac{1}{2^k} = \sbr{-\frac{1}{2^{k+1}}, \frac{1}{2^{k+1}}}.
	\]
	Возрастание $f_k$ на $A_{k,i}$ эквивалентно тому, что $f_k(a_{k,i}) = -1$. Тогда $f_k(b_{k,i}) = 1$, и по линейности
	\[
		f_k(a_{k+1,s}) = f_k\ps{-\frac{1}{2^{k+1}}} = -\frac{1}{2} = a_{1,j}.
	\]
	Отсюда следует и возрастание $f_{k+1}$ на $A_{k+1,s}$:
	\[
		f_{k+1}(a_{k+1,s}) = f_1(f_k(a_{k+1,s})) = f_1(a_{1,j}) = -1.
	\]
	
\end{proof}

Теперь с помощью этого знания можно оценить функцию $f$ в проколотой окрестности нуля:
\begin{proposition} \label{local_min}
	При всех $x \in \ps{0, \frac{1}{4}}$ $f(x) > 0$.
\end{proposition}

\begin{proof}
	Найдём такое $k \in \N$, что $x \in \rsi{\frac{1}{2^{k+1}}, \frac{1}{2^k}}$. По предыдущему утверждению при всех $l$ $f_l$ линейно возрастает от $-1$ до $1$ на $\sbr{-\frac{1}{2^l}, \frac{1}{2^l}}$. Тогда при $l \leq k$ имеется оценка $f_l(x) = 2^l x > \frac{2^l}{2^{k+1}}$. Применим её к частичной сумме:
	\[
		\sum_{l=1}^k \frac{f_l(x)}{2^l} > \sum_{l=1}^k \frac{1}{2^{k+1}} = \frac{k}{2^{k+1}}.
	\]
	Теперь посмотрим на остаток ряда:
	\[
		\sum_{l=k+1}^\infty \frac{f_l(x)}{2^l} \geq \sum_{l=k+1}^\infty \frac{-1}{2^l} = -\frac{1}{2^k}.
	\]
	Итого
	\[
		f(x) = \sum_{l=1}^\infty \frac{f_l(x)}{2^l} > \frac{k}{2^{k+1}} -\frac{1}{2^k} = \frac{k - 2}{2^{k+1}}.
	\]
	Если же $x \in \ps{0, \frac{1}{4}}$, то $k \geq 2$ и $f(x) > 0$, что и требовалось.
	
\end{proof}

Из определения $f_k$ или утверждения \ref{middle_intervals} следует, что при всех $k$ $f_k(0) = 0$. Соответственно, $f(0) = 0$, и по утверждению \ref{local_min} 0 является точкой строгого минимума функции $f$ на $\lsi{0, \frac{1}{4}}$. На этом теорема доказана.


\begin{thebibliography}{9}
	\bibitem{Wiki} Wikipedia contributors. \emph{Inflection point}. Wikipedia, The Free Encyclopedia, \url{https://en.wikipedia.org/w/index.php?title=Inflection_point&oldid=1243282412} (accessed May 13, 2025)
	\bibitem{Bruckner} A. Bruckner. \emph{Differentiation of Real Functions}. CRM Monograph series, 2nd ed., 1994. % ISBN 0-8218-6990-6
%	\bibitem{Volterra}
%	Vito Volterra. \emph{Sui Principii del Calcolo Integrale},
%%	[On the principles of the Integral Calculus]
%	Giornale di Mathematiche, XIX, pp. 333--372 (1881).
	\bibitem{Maximoff1940}
	И.М. Максимов. \emph{О преобразовании некоторых функций в точные производные}. Известия физико-математического общества при Казанском
	университете, т. XII. Сер. 3 (1940).
	\bibitem{Maximoff1943} I. Maximoff,
	\emph{On continuous transformation of some functions into an ordinary derivative}.
	Ann. Scuola Norm. Sup. Pisa., 12 (1943), pp. 147--160.
%	I. Maximoff. Sur In trallsformation continue de quelques fonctions en derivees exactes.
%	Bull. Soc. Phys. Math. Kazan, (3) 12 (19-10) 57-81. Russian; French summary.
	\bibitem{Weil76}
	Clifford E. Weil. \emph{On nowhere monotone functions}
	Proc. Amer. Math. Soc. 56 (1976), pp. 388--389.
	\bibitem{Kalyabin13}
	G. A. Kalyabin, \emph{New examples of Pompeiu functions}. Eurasian Math. J., 4:3 (2013), pp. 63--69.
	\bibitem{Stronska89} Ewa Stro\'nska. \emph{On the functions approximately- and quasi-continuous which are almost everywhere discontinuous.}
	Problemy Matematyczne, Zeszyt 10 (1989).
%	https://repozytorium.ukw.edu.pl/handle/item/6393

	\bibitem{Redkozubov} В.В. Редкозубов. \emph{Лекции по математическому анализу}.
	М.: МФТИ, 2024.

%	Ponce-Campuzano, J. C., & Maldonado-Aguilar, M. Á. (2015). Vito Volterra’s construction of a nonconstant function with a bounded, non-Riemann integrable derivative. BSHM Bulletin: Journal of the British Society for the History of Mathematics, 30(2), 143–152. https://doi.org/10.1080/17498430.2015.1010771
\end{thebibliography}
\end{document}